\documentclass{amsart}%
\usepackage[utf8]{inputenc}
\usepackage{amsmath, amssymb, amsthm, amsfonts}
\usepackage{tikz-cd}
\usepackage{amsmath}
\usepackage{amsfonts}
\usepackage{amssymb}
\usepackage{graphicx}%
\providecommand{\U}[1]{\protect\rule{.1in}{.1in}}
\newtheorem{definition}{Definition}[section]
\newtheorem{example}{Example}[section]
\newtheorem{remark}{Remark}[section]
\newtheorem{theorem}{Theorem}[section]
\newtheorem{corollary}{Corollary}[section]
\newtheorem{lemma}{Lemma}[section]
\newtheorem{proposition}{Proposition}[section]
\begin{document}
\title[Quasi-sdf-absorbing ideals]{Quasi sdf-absorbing ideals in commutative rings}
\date{}
\author{Violeta Leoreanu-Fotea$^{*} $}
\address{Faculty of Mathematics, Al.I. Cuza University, Bd. Carol I, No. 11, 700506
Ia\c{s}i, Romania.}
\email{foteavioleta@gmail.com}
\author{Ece Yetkin Celikel}
\address{Department of Software Engineering, Hasan Kalyoncu University, Gaziantep, T\"{u}rkiye}
\email{ece.celikel@hku.edu.tr, yetkinece@gmail.com}
\author{Tarik Arabaci}
\address{Department of Basic Science, Faculty of Engineering and Architecture, Istanbul
Geli\c{s}im University, Istanbul, T\"{u}rkiye}
\email{tarabaci@gelisim.edu.tr}
\author{Unsal Tekir}
\address{Department of Mathematics, Marmara University, Istanbul, T\"urkiye}
\email{utekir@marmara.edu.tr}
\thanks{*Corresponding author: foteavioleta@gmail.com}
\keywords{Square-difference factor absorbing ideal, sdf-absorbing primary ideal,
quasi-sdf-absorbing ideal}
\subjclass{13A15, 13C05, 13C13}

\begin{abstract}
This paper introduces and studies quasi sdf-absorbing ideals as a
generalization of sdf-absorbing ideals. We investigate the stability of this
property under various constructions, including localization, surjective
images, Nagata idealizations, and amalgamations. We establish conditions under
which the radical of such ideals is prime and discuss a specific class of
rings where quasi sdf-absorption implies the sdf-absorbing primary property.
The study concludes with a classification of these ideals in $\mathbb{Z}$ and
examples distinguishing them from related ideal classes.

\end{abstract}
\maketitle

\section{Introduction}

The study of ideal-theoretic generalizations of prime and primary ideals has
attracted considerable attention in recent years. Such generalizations often
arise by weakening the defining conditions of classical notions while
preserving important structural properties, a framework thoroughly discussed
in classical texts like \cite{AM1969} and \cite{Lang}. A notable direction in
this area was initiated by A. Badawi \cite{B2007}, who introduced the concept
of $2$-absorbing ideals, extending the classical notion of prime ideals. This
work was further expanded to the primary case in \cite{BTY}, where
$2$-absorbing primary ideals were characterized.

Throughout this paper, all rings are assumed to be commutative with identity.
Motivated by similar ideas, Anderson, Badawi, and Coykendall \cite{AB2024}
recently introduced the class of square-difference factor absorbing ideals
(briefly, \emph{sdf-absorbing ideals}) of rings.. An ideal $I$ of a ring $R$
is called sdf-absorbing if for all non-zero $a,b\in R$, $a^{2}-b^{2}\in I$
implies $a-b\in I\;\text{or}\;a+b\in I.$ A natural extension of this notion,
which involves the radical of the ideal, was proposed by Khashan, Yetkin
Celikel, and Tekir \cite{KhashanCelikel}. They defined an ideal $I$ of a ring
$R$ to be \emph{sdf-absorbing primary} if for all $a,b\in R$, $a^{2}-b^{2}\in
I$ implies $a+b\in\sqrt{I}\;\text{or}\;a-b\in I.$ A fundamental structural
result established in \cite{KhashanCelikel} states that if $I$ is an
sdf-absorbing primary ideal, then its radical $\sqrt{I}$ is necessarily
sdf-absorbing. However, the converse of this implication does not hold in
general (See Example 2.1 (i)).

Inspired by this connection, the present paper is devoted to the study of
\emph{quasi sdf-absorbing ideals}, defined as those ideals $I$ whose radicals
$\sqrt{I}$ are sdf-absorbing. This class provides a flexible framework for
analyzing the "absorbing" nature of an ideal in rings that are not necessarily
reduced. We investigate the behavior of these ideals under various
constructions, including the amalgamation extension, a structure recently
surveyed in \cite{ElKhalfi-Kim-Mahdou}. Specifically, we provide a complete
characterization of quasi sdf-absorbing ideals in the ring of integers
$\mathbb{Z}$ and analyze their transfer properties in idealizations and
amalgamated duplications. Furthermore, we identify and explore a particular
class of rings where the quasi sdf-absorbing property and the sdf-absorbing
primary property coincide, examining how this relationship is preserved under
fundamental ring-theoretic operations.

\section{Properties of Quasi sdf-absorbing ideals}

In 1946, Fuchs introduced the concept of quasi-primary ideals. According to
\cite{F} a proper ideal is said to be quasi-primary if its radical is
prime.\ More generally, in 1995, Jayaram and Johnson \cite{J} studied this
notion in a more general lattice-theoretic framework. Motivated by these
concepts, we introduce in this work a natural generalization, namely the class
of quasi sdf-absorbing ideals. The main result of this study is presented below.

\begin{definition}
Let $R$ be a ring and $I$ a proper ideal of $R.$ We say that $I$ is
\emph{quasi sdf-absorbing} if its radical $\sqrt{I}$ is a sdf-absorbing ideal
of $R.$
\end{definition}

It follows immediately that every quasi-primary and sdf-absorbing ideals are
quasi sdf-absorbing. Nevertheless, the converse implication does not hold in
general, as illustrated by the following example.

\begin{example}
\ 
\end{example}

\begin{enumerate}
\item Consider the ideal $I=4q^{m}\mathbb{Z}$ of $\mathbb{Z}$, where $q$ is an
odd prime and $m\geq1$. Then $I$ is a quasi sdf-absorbing ideal, but it is not
an sdf-absorbing primary ideal. Indeed, $\sqrt{I}=2q\mathbb{Z},$ which is an
sdf-absorbing ideal, by \cite[Example 2.8 (a)]{AB2024}, therefore, $I$ is
quasi sdf-absorbing. However, $I$ is not an sdf-absorbing primary ideal.
Consider $a=2q^{m}+1,\ b=1.$ Then $a^{2}-b^{2}=(2q^{m}+1)^{2}-1^{2}%
=4q^{m}(q^{m}+1)\in I.$ However, $a-b=2q^{m}\notin I$ and for any positive
integer $k\geq1$, $(a+b)^{k}=(2q^{m}+2)^{k}=2^{k}(q^{m}+1)^{k}\notin I.$
Hence, $I$ is not an sdf-absorbing primary ideal.

\item Let $K$ be a field, let $q$ be an odd prime, and $m\geq1$. Consider the
ideal $I=4q^{m}K[x_{1},x_{2},\dots,x_{n}]$ of $K[x_{1},x_{2},\dots,x_{n}].$
Then $I$ is quasi sdf-absorbing, but it is not quasi-primary. Indeed, the
radical is $\sqrt{I}=2qK[x_{1},\dots,x_{n}],$ which is sdf-absorbing by
\cite[Corollary 4.3]{AB2024}. Hence $I$ is quasi sdf-absorbing. The ideal $I$
is not quasi-primary, since $\sqrt{I}$ is not prime.

\item In any ring $R$ with characteristic $2$ (in particular, if $R$ is a
Boolean ring), every proper ideal of $R$ is quasi sdf-absorbing. Indeed, let
$I$ be a proper ideal of $R$ with $char(R)=2$. Suppose that $a^{2}-b^{2}%
\in\sqrt{I}$, for some $0\neq a,b\in R.$ Then $(a-b)^{2}\in\sqrt{I}$ which
implies $a-b\in\sqrt{I}$, so $\sqrt{I}$ is sdf-absorbing. Thus $I$ is a quasi
sdf-absorbing ideal of $R.$
\end{enumerate}

Next, we verify that if $2$ is a unit element in a ring, then the quasi
sdf-absorbing and quasi primary ideals coincide.

\begin{proposition}
Let $R$ be a ring such that $2$ is a unit element. Then, a proper ideal of $R$
is quasi sdf-absorbing if and only if it is quasi primary.
\end{proposition}

\begin{proof}
Let $I$ be a quasi sdf-absorbing ideal of $R$ and $a,b\in R$ such that
$ab\in\sqrt{I}$. Then $(a+b)^{2}-(a-b)^{2}=4ab\in\sqrt{I}$. If $a\neq b$ and
$a\neq-b$, $2a=(a+b)+(a-b)\in\sqrt{I}$ or $2b=(a+b)-(a-b)\in\sqrt{I}.$ Since
$2\in U(R)$, we have $a\in\sqrt{I}$ or $b\in\sqrt{I}$. Now assume that $a=b$
(respectively, $a=-b$). Then, we have $a^{2}\in\sqrt{I}$ (respectively,
$-a^{2}\in\sqrt{I}$) which implies $a\in\sqrt{I}$. Thus, $\sqrt{I}$ is prime,
and so $I$ is a quasi primary ideal of $R.$ The converse part is straightforward.
\end{proof}

The following observation will be useful in the sequel, as it gives another
equivalent form of the definition.

\begin{remark}
\label{rr}Let $I$ be a proper ideal of a ring $R.$ Then the following
assertions are equivalent:
\end{remark}

\begin{enumerate}
\item $I$ is a quasi sdf-absorbing ideal of $R.$

\item For $0\neq a,b\in R$ such that $a^{2}-b^{2}\in\sqrt{I}$, we have
$a-b\in\sqrt{I}$ or $a+b\in\sqrt{I}$.

\item For $a,b\in R$ such that $a^{2}-b^{2}\in\sqrt{I}$, we have $a-b\in
\sqrt{I}$ or $a+b\in\sqrt{I}$.
\end{enumerate}

\begin{proof}
(1)$\Rightarrow$(2) and (3)$\Rightarrow$(1) are clear by the definition.
Hence, to complete the proof, it is sufficient to show that (2)$\Rightarrow
$(3). Suppose $a^{2}-b^{2}\in\sqrt{I}$ for some $a,b\in R.$ If $a,b$ are
nonzero, we are done by (2). Without loss of generality, assume that $b=0.$
Then $a^{2}=a^{2}-b^{2}\in\sqrt{I}$ implies $a\in\sqrt{I}$, and hence both of
$a-b$ and $a+b$ belongs to $\sqrt{I}$.
\end{proof}

In the following, analogous to the results obtained in \cite[Theorems 3.7,
4.1, 4.5, 4.6]{AB2024} we establish their counterparts in the broader context
of quasi sdf-absorbing ideals.

\begin{theorem}
\label{Th} Let $R$ be a ring and $I=\bigcap_{j=1}^{n}P_{j}$ be an intersection
of comaximal prime ideals in $R$. Then $I$ is quasi sdf-absorbing if and only
if at most one factor $R/P_{j}$ has characteristic different from $2$.
\end{theorem}

\begin{proof}
Note that $\sqrt{I}=\sqrt{\cap_{j=1}^{n}P_{j}}=\cap_{j=1}^{n}\sqrt{P_{j}}$ and
$\sqrt{P_{j}}=P_{j}$ 's are also comaximal prime ideals for $j=1,...,n.$ Then,
the claim is clear by \cite[Theorem 4.1]{AB2024}.
\end{proof}

\begin{theorem}
Let $n\in\mathbb{N}$. Then $n\mathbb{Z}$ is a quasi-sdf-absorbing ideal if and
only if $n$ has at most one odd prime divisor, i.e., $n=q^{m}$ for some
$m\geq1$ or $n=2^{k}q^{m}$ for some $k\geq1,m\geq0$ and a odd prime integer
$q$.
\end{theorem}

\begin{proof}
By definition, $n\mathbb{Z}$ is a quasi sdf-absorbing ideal of $\mathbb{Z}$ if
and only if its radical $\sqrt{n\mathbb{Z}}$ is an sdf-absorbing ideal. We
have $\sqrt{n\mathbb{Z}}=p_{1}p_{2}\dots p_{m}\mathbb{Z}$, where $p_{1}%
,p_{2},\dots,p_{m}$ are distinct prime divisors of $n$. The ideal $p_{1}%
p_{2}\dots p_{m}\mathbb{Z}$ can be written as an intersection of pairwise
comaximal prime ideals:
\[
p_{1}p_{2}\dots p_{m}\mathbb{Z}=p_{1}\mathbb{Z}\cap p_{2}\mathbb{Z}\cap
\dots\cap p_{m}\mathbb{Z}.
\]
According to Theorem \ref{Th}, an intersection of comaximal prime ideals is
sdf-absorbing if and only if at most one of the factor rings $\mathbb{Z}%
/p_{j}\mathbb{Z}$ has characteristic different from $2$. Note that
$\text{char}(\mathbb{Z}/p_{j}\mathbb{Z})=p_{j}$. Therefore, the condition is
satisfied if and only if $p_{1}p_{2}\dots p_{m}$ has at most one odd prime
divisor. This means that $n$ must be of the form $n=q^{m}$ for some $m\geq1$
or $n=2^{k}q^{m}$ for some $k\geq1,m\geq0.$
\end{proof}

\begin{theorem}
Let $I$ be a proper ideal of a ring $R$ and let $\sqrt{I}=\bigcap_{j=1}%
^{n}P_{j}$ be a finite intersection of pairwise comaximal prime ideals $P_{j}$
of $R$. Then the following are equivalent:

\begin{enumerate}
\item $I$ is quasi sdf-absorbing in $R$.

\item $I[x]$ is quasi sdf-absorbing in the polynomial ring $R[x]$.
\end{enumerate}
\end{theorem}

\begin{proof}
Let $I$ be a quasi sdf-absorbing ideal of $R.$ Then $\sqrt{I}$ is a
sdf-absorbing ideal. From \cite[Theorem 4.5]{AB2024}, $\sqrt{I}[X]$ is a
sdf-absorbing ideal of $R[X]$. Since $\sqrt{I[X]}=\sqrt{I}[X]$ is
sdf-absorbing, we have that $I[X]$ is quasi sdf-absorbing in $R[X].$ The
converse is clear by using the same argument.
\end{proof}

\begin{theorem}
Let $R$ be a ring and $I$ a proper ideal of $R$. Then $(I,X)$ is a quasi
sdf-absorbing ideal of $R[X]$ if and only if $I$ is a quasi sdf- absorbing
ideal of $R$. In particular, $(X)$ is a quasi sdf- absorbing ideal of $R[X]$
if and only if $\sqrt{0}$ is a quasi sdf-absorbing ideal of $R.$
\end{theorem}

\begin{proof}
($\Rightarrow$) Assume that $(I,X)$ is a quasi sdf-absorbing ideal of $R[X]$.
Let $a,b\in R$ such that $a^{2}-b^{2}\in\sqrt{I}.$ Then $a^{2}-b^{2}\in
\sqrt{(I,X)}$, since $\sqrt{I}\subseteq(\sqrt{I},X)=\sqrt{(I,X)}$. Viewing
$a,b$ as constant polynomials, we get: $a+b\in\sqrt{(I,X)}\ \text{or}%
\ a-b\in\sqrt{(I,X)}.$ Hence: $a+b\in\sqrt{I}\ \text{or}\ a-b\in\sqrt{I}.$
Thus $I$ is quasi sdf-absorbing.

($\Leftarrow$) Assume that $I$ is quasi sdf-absorbing. Let $f,g\in R[X]$ such
that $f^{2}-g^{2}\in\sqrt{(I,X)}=(\sqrt{I},X).$ Write: $f=a+Xf_{1}%
,\ g=b+Xg_{1}.$ Then the constant term of $f^{2}-g^{2}$ is: $a^{2}-b^{2}%
\in\sqrt{I}.$ By hypothesis: $a+b\in\sqrt{I}\ \text{or}\ a-b\in\sqrt{I}.$
Hence: $f+g\in(\sqrt{I},X)=\sqrt{(I,X)}$ or $f-g\in(\sqrt{I},X)=\sqrt{(I,X)}.$
Therefore $(I,X)$ is quasi sdf-absorbing.
\end{proof}

\section{Behavior of Quasi sdf-absorbing ideals under Ring Constructions}

In this section, we examine the behavior of quasi sdf-absorbing ideals with
respect to several fundamental ring constructions, including ring
homomorphisms, localizations, quotient rings, cartesian products,
idealizations, and amalgamated rings.

\begin{theorem}
\label{T} Let $\varphi:R\rightarrow S$ be a ring homomorphism.
\end{theorem}

\begin{enumerate}
\item Let $\varphi$ be surjective and $I$ a quasi sdf-absorbing ideal of $R$.
If $\ker\varphi\subseteq\sqrt{I}$, then $\varphi(I)$ is a quasi sdf-absorbing
ideal of $S$.

\item If $J$ is a quasi-sdf-absorbing ideal of $S$, then $\varphi^{-1}(J)$ is
a quasi sdf-absorbing ideal of $R$.
\end{enumerate}

\begin{proof}
(1) Let $0\neq s,t\in S$ such that $s^{2}-t^{2}\in\sqrt{\varphi(I)}$. Since
$\varphi$ is surjective, there exist nonzero $a,b\in R$ such that
$\varphi(a)=s$ and $\varphi(b)=t$. Then $\varphi(a^{2}-b^{2})\in\varphi
(\sqrt{I}),$ and hence there exists some $r\in\sqrt{I}$ such that
$\varphi(a^{2}-b^{2})=\varphi(r)$. This is equivalent to: $a^{2}-b^{2}%
-r\in\ker\varphi.$ By the hypothesis $\ker\varphi\subseteq\sqrt{I}$, we have
$a^{2}-b^{2}\in\sqrt{I}$. Since $I$ is quasi sdf-absorbing, $\sqrt{I}$ is an
sdf-absorbing ideal and $a,b\neq0$, which imply that $a-b\in\sqrt{I}$ or
$a+b\in\sqrt{I}$. Thus, we have: $s-t=\varphi(a)-\varphi(b)\in\varphi(\sqrt
{I})=\sqrt{\varphi(I)}\ \text{or}\ s+t=\varphi(a)+\varphi(b)\in\varphi
(\sqrt{I})=\sqrt{\varphi(I)}.$ Hence $s-t\in\sqrt{\varphi(I)}$ or $s+t\in
\sqrt{\varphi(I)}$. Therefore, $\varphi(I)$ is a quasi sdf-absorbing ideal of
$S$.

(2) We check that $\sqrt{\varphi^{-1}(J)}$ is an sdf-absorbing ideal of $R$.
We use the standard property of ring homomorphisms: $\sqrt{\varphi^{-1}%
(J)}=\varphi^{-1}(\sqrt{J}).$ Let $a,b\in R$ such that $a^{2}-b^{2}\in
\varphi^{-1}(\sqrt{J})$. This implies that: $\varphi(a^{2}-b^{2}%
)=\varphi(a)^{2}-\varphi(b)^{2}\in\sqrt{J}.$ If $\varphi(a)=0$ then
$\varphi(b)^{2}\in\sqrt{J},$ and so $\varphi(b)\in\sqrt{J}$ implies
$\varphi(a)+\varphi(b)=\varphi(a+b)\in\sqrt{J}$. Thus $a+b\in\varphi
^{-1}(\sqrt{J})=\sqrt{\varphi^{-1}(J)}.$ Similarly $\varphi(b)=0$, causes
$a+b\in\sqrt{\varphi^{-1}(J)}.$ So we can suppose $\varphi(a)\neq0\neq
\varphi(b).$ Since $\sqrt{J}$ is an sdf-absorbing ideal, we obtain:
$\varphi(a)-\varphi(b)\in\sqrt{J}\ \text{or}\ \varphi(a)+\varphi(b)\in\sqrt
{J},$ whence $a-b\in\varphi^{-1}(\sqrt{J})$ or $a+b\in\varphi^{-1}(\sqrt{J})$.
It follows that $\sqrt{\varphi^{-1}(J)}$ is an sdf-absorbing ideal, that is
$\varphi^{-1}(J)$ is a quasi sdf-absorbing ideal of $R$.
\end{proof}

As a consequence of Theorem \ref{T}, we have the following result.

\begin{corollary}
Let $R$ be a ring and $I$ a proper ideal of $R.$
\end{corollary}

\begin{enumerate}
\item Let $R\subseteq S$ be a ring extension. If $I$ is a quasi sdf-absorbing
ideal of $S$, then $I\cap R$ is a quasi sdf-absorbing ideal of $R$.

\item Let $J$ be an ideal of $R$ contained in $I.$ Then $I$ is a quasi
sdf-absorbing ideal of $R$ if and only if $I/J$ is a quasi sdf-absorbing ideal
of $R/J$.
\end{enumerate}

\begin{proposition}
Let $R$ be a ring, $S\subseteq R$ a multiplicative set, and $I$ a
quasi-sdf-absorbing ideal of $R$, such that $I\cap S=\emptyset.$ Then the
localized ideal $S^{-1}I$ is quasi sdf-absorbing in $S^{-1}R$.
\end{proposition}

\begin{proof}
First, note that $\sqrt{I}\cap S=\emptyset$. Indeed, if there exists
$s\in\sqrt{I}\cap S$, then $s^{n}\in I\cap S$ for some positive integer $n$, a
contradiction. Now, suppose that $I$ is a quasi sdf-absorbing ideal of $R.$
Then $\sqrt{I}$ is sdf-absorbing,$\ $hence $S^{-1}\sqrt{I}=\sqrt{S^{-1}I}$ is
sdf-absorbing ideal of $S^{-1}R$ by \cite[Theorem 2.9]{AB2024}. Thus,
$S^{-1}I$ is a quasi sdf-absorbing ideal of $S^{-1}R$.
\end{proof}

An ideal that is not quasi sdf-absorbing in $R$ may have a quasi sdf-absorbing
localization in $S^{-1}R$.

\begin{example}
Consider the ideal $I=15\mathbb{Z}$ in $R=\mathbb{Z}$ and the multiplicative
set $S=\{5^{n}\mid n\geq0\}$. First, we show that $I$ is not quasi
sdf-absorbing in $\mathbb{Z}$. Since $I$ is a radical ideal, $\sqrt
{I}=15\mathbb{Z}$. Let $a=4$ and $b=1$. Then: $a^{2}-b^{2}=16-1=15\in
15\mathbb{Z},$ but $a-b=3\notin15\mathbb{Z}$ and $a+b=5\notin15\mathbb{Z}$.
Thus, $\sqrt{I}$ is not sdf-absorbing. Now, consider the localized ideal
$S^{-1}I$ in $S^{-1}\mathbb{Z}$. Since $5\in S$ is a unit in the localized
ring, we have: $S^{-1}(15\mathbb{Z})=S^{-1}(3\mathbb{Z}).$ The radical of this
ideal is $\sqrt{S^{-1}I}=S^{-1}(3\mathbb{Z})$, which is a maximal (and thus
prime) ideal in $S^{-1}\mathbb{Z}$. Since every prime ideal is sdf-absorbing,
it follows that $S^{-1}I$ is a quasi-sdf-absorbing ideal in $S^{-1}\mathbb{Z}$.
\end{example}

Let $I$ be an ideal of a ring $R$ and $a\in R\backslash I$. Then the inclusion
$\sqrt{(I:a)}\subseteq(\sqrt{I}:a)$ always holds. But, the reverse inclusion
$(\sqrt{I}:a)\subseteq\sqrt{(I:a)}$ does not hold in general. Let
$R=\mathbb{Z}$, $I=(4)$, and $a=2$. Then $(I:a)=\{x\in\mathbb{Z}\mid
2x\in(4)\}=(2),$ $\sqrt{(I:a)}=(2)$. On the other hand, since $\sqrt{I}%
=\sqrt{(4)}=(2)$ we have $(\sqrt{I}:a)=((2):2)=\mathbb{Z}.$ This shows that
the inclusion $(\sqrt{I}:a)\subseteq\sqrt{(I:a)}$ does not hold in general. If
$a$ is an idempotent element of $R$ (i.e., $a^{2}=a$), then we have the
equality $\sqrt{(I:a)}=(\sqrt{I}:a)$. Indeed, let $x\in(\sqrt{I}:a)$, so
$xa\in\sqrt{I}$. Then $(xa)^{n}\in I$ for some $n\geq1$ whence $x^{n}\in
(I:a)$, hence $x\in\sqrt{(I:a)}$. In particular, if $R$ is a boolean ring,
then the equality always holds for all $a\in R$.

\begin{proposition}
Let $I$ be a quasi sdf-absorbing ideal of a ring $R$, such that $\sqrt
{(I:a)}=(\sqrt{I}:a).$ Then, for any $a\in R\setminus I$, the colon ideal
$(I:a)=\{r\in R\mid ra\in I\}$ is also quasi sdf-absorbing.
\end{proposition}

\begin{proof}
Suppose that $I$ is a quasi sdf-absorbing ideal of $R$. Let $r,s\in R$ such
that $r^{2}-s^{2}\in\sqrt{(I:a)}$, which means that $(r^{2}-s^{2})a\in\sqrt
{I}.$ Hence, $(ra)^{2}-(sa)^{2}=(r^{2}-s^{2})a^{2}\in\sqrt{I}$ which implies
$ra-sa\in\sqrt{I}\ \text{or}\ ra+sa\in\sqrt{I}$ by Remark \ref{rr}. Then
$r-s\in(\sqrt{I}:a)=\sqrt{(I:a)}\ \text{or}\ r+s\in(\sqrt{I}:a)=\sqrt{(I:a)}.$
Thus, $(I:a)$ is quasi sdf-absorbing.by again Remark \ref{rr}.
\end{proof}

\begin{proposition}
Let $I$ and $J$ be quasi-sdf-absorbing ideals of a ring $R$. If $\sqrt
{I}=\sqrt{J}$, then $I\cap J$ is quasi-sdf-absorbing.
\end{proposition}

\begin{proof}
We have $\sqrt{I\cap J}=\sqrt{I}\cap\sqrt{J}=\sqrt{I}$. By hypothesis,
$\sqrt{I}$ is an sdf-absorbing ideal, hence $\sqrt{I\cap J}$ is sdf-absorbing.
Therefore, $I\cap J$ is a quasi sdf-absorbing ideal of $R$.
\end{proof}

In analogy with \cite[Theorem 4.12]{AB2024} we develop parallel results for
quasi sdf-absorbing ideals.

\begin{theorem}
Let $I_{1},$ $I_{2}$ be proper ideals of $R_{1},$ $R_{2}$, respectively. Then
$I_{1}\times I_{2}$ is a quasi sdf-absorbing ideal of $R_{1}\times R_{2}$ if
and only if $I_{1},$ $I_{2}$ are quasi sdf-absorbing ideals of $R_{1},$
$R_{2}$, respectively and $2\in\sqrt{I_{1}}$ or $2\in\sqrt{I_{2}}$.
\end{theorem}

\begin{proof}
Let $R=R_{1}\times R_{2}$ and $I=I_{1}\times I_{2}$. We use the fact that
$\sqrt{I}=\sqrt{I_{1}}\times\sqrt{I_{2}}$. $(\implies)$ Assume $\sqrt{I}$ is
an sdf-absorbing ideal of $R$. It is straightforward to show that its
components $\sqrt{I_{1}}$ and $\sqrt{I_{2}}$ are sdf-absorbing ideals of
$R_{1}$ and $R_{2}$, respectively, hence $I_{1},$ $I_{2}$ are quasi
sdf-absorbing. To show that $2\in\sqrt{I_{1}}$ or $2\in\sqrt{I_{2}}$, let
$a=(1,1)$ and $b=(1,-1)$ in $R$. Then $a^{2}-b^{2}=(1-1,1-1)=(0,0)\in\sqrt{I}%
$. Since $\sqrt{I}$ is sdf-absorbing, we must have $a+b=(2,0)\in\sqrt{I}$ or
$a-b=(0,2)\in\sqrt{I}$. This implies $2\in\sqrt{I_{1}}$ or $2\in\sqrt{I_{2}}$.

$(\impliedby)$ Assume $I_{1},I_{2}$ are quasi sdf-absorbing and, without loss
of generality, $2\in\sqrt{I_{1}}$. Let $a=(a_{1},a_{2}),b=(b_{1},b_{2})\in R$
such that $a^{2}-b^{2}\in\sqrt{I}$. This means $a_{1}^{2}-b_{1}^{2}\in
\sqrt{I_{1}}$ and $a_{2}^{2}-b_{2}^{2}\in\sqrt{I_{2}}$. Since $\sqrt{I_{1}}$
is sdf-absorbing and $2\in\sqrt{I_{1}}$, by \cite[Theorem 2.5]{AB2024}, we
have both $a_{1}+b_{1}\in\sqrt{I_{1}}$ and $a_{1}-b_{1}\in\sqrt{I_{1}}$. At
the same time, since $I_{2}$ is quasi sdf-absorbing, we have $a_{2}+b_{2}%
\in\sqrt{I_{2}}$ or $a_{2}-b_{2}\in\sqrt{I_{2}}$. If $a_{2}+b_{2}\in
\sqrt{I_{2}}$, then $(a_{1}+b_{1},a_{2}+b_{2})=a+b\in\sqrt{I}$. If
$a_{2}-b_{2}\in\sqrt{I_{2}}$, then $(a_{1}-b_{1},a_{2}-b_{2})=a-b\in\sqrt{I}$.
In both cases, $\sqrt{I}$ satisfies the sdf-absorbing condition. Thus,
$I_{1}\times I_{2}$ is quasi sdf-absorbing.
\end{proof}

The condition $2\in\sqrt{I_{1}}$ or $2\in\sqrt{I_{2}}$ in the above Theorem is
necessary. The following example is to illustrate that if this condition is
not satisfied, the product ideal $I_{1}\times I_{2}$ fails to be quasi
sdf-absorbing even if both $I_{1}$ and $I_{2}$ are quasi sdf-absorbing.

\begin{example}
Consider $I_{1}=3\mathbb{Z}$ and $I_{2}=5\mathbb{Z}$ in the ring
$R=\mathbb{Z}\times\mathbb{Z}$. Then, $I_{1}$ and $I_{2}$ are quasi
sdf-absorbing in $\mathbb{Z}$ since they are prime. However, $2\notin
\sqrt{I_{1}}=3\mathbb{Z}$ and $2\notin\sqrt{I_{2}}=5\mathbb{Z}$. Put $x=(4,1)$
and $y=(1,4)$ in $\mathbb{Z}\times\mathbb{Z}$. We have $x^{2}-y^{2}%
=(15,-15)\in3\mathbb{Z}\times5\mathbb{Z}=\sqrt{I_{1}\times I_{2}},$ but
$x-y=(3,-3)\notin3\mathbb{Z}\times5\mathbb{Z}$, $x+y=(5,5)\notin
3\mathbb{Z}\times5\mathbb{Z}$. Thus, $I_{1}\times I_{2}$ is not a quasi
sdf-absorbing ideal.
\end{example}

Let $R$ be a commutative ring and $M$ an $R$-module. The \emph{idealization}
of $M$ in $R$ is the ring $R \ltimes M = \{ (r,m) \mid r \in R, m \in M \} $
with operations
\[
(r_{1}, m_{1}) + (r_{2}, m_{2}) = (r_{1} + r_{2}, m_{1} + m_{2}), \quad(r_{1},
m_{1})(r_{2}, m_{2}) = (r_{1} r_{2}, r_{1} m_{2} + r_{2} m_{1}).
\]

Let $I\subsetneq R$ be an ideal. Consider the ideal in the idealization
$I\ltimes M=\{(r,m)\mid r\in I,m\in M\}.$ Let $R\ltimes M$ be the idealization
of an $R$-module $M$. For any ideal $I$ of $R$, let $J=I\ltimes M$. According
to \cite[Theorem 3.2]{Anderson2009}, the radical of an ideal in the
idealization $R\ltimes M$ is given by:
\[
\sqrt{I\ltimes M}=\sqrt{I}\ltimes M
\]
where $I$ is an ideal of $R$ and $M$ is an $R$-module. This follows from the
fact that $(r,m)^{n}=(r^{n},nr^{n-1}m)$, so $(r,m)\in\sqrt{I\ltimes M}$ if and
only if $r\in\sqrt{I}$.

\begin{proposition}
Let $I$ be a proper ideal of a ring $R$ and $M$ an $R$-module. Then $I\ltimes
M$ is a quasi sdf-absorbing ideal of $R\ltimes M$ if and only if $I$ is a
quasi sdf-absorbing ideal of $R$.
\end{proposition}

\begin{proof}
Suppose that $I$ is a quasi sdf-absorbing ideal of $R.$ Let $x=(a,m_{1})$ and
$y=(b,m_{2})$ be elements of $R\ltimes M$ such that $x^{2}-y^{2}\in
\sqrt{I\ltimes M}$. This implies: $(a^{2}-b^{2},2am_{1}-2bm_{2})\in\sqrt
{I}\ltimes M.$ The first component satisfies $a^{2}-b^{2}\in\sqrt{I}$. Since
$I$ is a quasi-sdf-absorbing ideal of $R$, we have $a-b\in\sqrt{I}$ or
$a+b\in\sqrt{I}$. Hence, $(a,m_{1})-(b,m_{2})=(a-b,m_{1}-m_{2})\in\sqrt
{I}\ltimes M$ or $(a,m_{1})+(b,m_{2})=(a+b,m_{1}+m_{2})\in\sqrt{I}\ltimes M$.
Conversely, suppose that $I\ltimes M$ is a quasi sdf-absorbing ideal of
$R\ltimes M.$ Let $a,b\in R$ such that $a^{2}-b^{2}\in I$. Then $(a,0)^{2}%
-(b,0)^{2}\in I\ltimes M$ which implies $(a,0)-(b,0)\in\sqrt{I\ltimes M}%
=\sqrt{I}\ltimes M$ or $(a,0)+(b,0)\in\sqrt{I\ltimes M}=\sqrt{I}\ltimes M$.
Thus, $a-b\in\sqrt{I}$ or $a+b\in\sqrt{I}$, and so $I$ is a quasi
sdf-absorbing ideal of $R$ by Remark \ref{rr}.
\end{proof}

Let $R$ and $S$ be commutative rings, $J $ a proper ideal of $S$ and $\varphi:
R \to S$ a ring homomorphism. The \emph{amalgamation} of $R$ with $S$ along
$J$ with respect to $\varphi$ is the subring of $R \times S$ defined by $R
\bowtie^{\varphi}J := \{ (r, \varphi(r) + j) \mid r \in R, j \in J \}. $

We consider the corresponding ideal in the amalgamation by
\[
I\bowtie^{\varphi}J:=\{(r,\varphi(r)+j)\mid r\in I,\text{ }j\in J\}\subseteq
R\bowtie^{\varphi}J.
\]

By \cite[Proposition 2.1 (2)]{DAnna2009}, we have the following canonical
isomorphism:
\[
\frac{R\bowtie^{\varphi}J}{I\bowtie^{\varphi}J}\cong R/I
\]
This structural property suggests that the radical of $I\bowtie^{\varphi}J$
should reflect the radical of $I$ in $R$. The radical of the ideal
$I\bowtie^{f}J$ is given by the amalgamation of the radicals \cite[Lemma
9]{AE}:
\[
\sqrt{I\bowtie^{\varphi}J}=\sqrt{I}\bowtie^{\varphi}J
\]

\begin{remark}
Let $I\subsetneq R$ be a quasi sdf-absorbing ideal. Let $X=(a,\varphi
(a)+j_{1})$ and $Y=(b,\varphi(b)+j_{2})$ be elements of $R\bowtie^{\varphi}J$
such that $X^{2}-Y^{2}\in\sqrt{I\bowtie^{\varphi}J}$. This implies that the
first component satisfies: $a^{2}-b^{2}\in\sqrt{I}.$ Since $I$ is a quasi
sdf-absorbing ideal of $R$, we have $a-b\in\sqrt{I}$ or $a+b\in\sqrt{I}$. In
the amalgamated ring, if $a-b\in\sqrt{I}$, then for any elements in the second
component, we have: $X-Y=(a-b,\varphi(a-b)+(j_{1}-j_{2}))\in\sqrt{I}%
\bowtie^{\varphi}J.$ By using the same logic, we conclude $X+Y\in\sqrt
{I}\bowtie^{\varphi}J$ if $a+b\in\sqrt{I}$. Thus, the quasi sdf-absorbing
property of $I\bowtie^{\varphi}J$ is entirely determined by the quasi
sdf-absorbing property of $I$ in $R$. We conclude that $I\bowtie^{\varphi}J$
is a quasi sdf-absorbing ideal of $R\bowtie^{\varphi}J$ whenever $I$ is quasi
sdf-absorbing in $R$.
\end{remark}

\section{On the class of rings satisfying condition $(\ast)$}

In this section, we characterize the class of rings in which the distinction
between the quasi sdf-absorbing property and the sdf-absorbing primary
property vanishes. We define this class as follows:

\begin{definition}
A ring $R$ is said to satisfy condition $(\ast)$ if every quasi sdf-absorbing
ideal of $R$ is an sdf-absorbing primary ideal.
\end{definition}

For a ring $R$ satisfying $(\ast)$, the structural relationship between an
ideal and its radical becomes particularly clear: for any proper ideal $I$ of
$R$, the following statements are equivalent:

\begin{enumerate}
\item $I$ is an sdf-absorbing primary ideal;

\item $\sqrt{I}$ is an sdf-absorbing ideal.
\end{enumerate}

We list below several important classes of rings that satisfy condition
$(\ast).$

\begin{example}
\label{estar} \ 
\end{example}

\begin{enumerate}
\item \text{Von Neumann regular rings:} In these rings, every ideal is
semiprime ($I=\sqrt{I}$). Thus, any quasi sdf-absorbing ideal is its own
radical and sdf-absorbing, making $(\ast)$ trivial.

\item \text{Rings of characteristic 2:} In such rings, every proper ideal of
$R$ is sdf-absorbing primary.

\item If $2$ is invertible, the notion of an sdf-absorbing ideal coincides
with that of a prime ideal and an sdf-absorbing primary ideal is exactly a
primary ideal. Consequently, in this setting, Condition $(\ast)$ reflects the
classical property that ideals with prime radicals are primary.
\end{enumerate}

\begin{theorem}
Let $R$ be a $0$-dimensional Noetherian ring such that $2\in U(R)$. Then $R$
satisfies the condition $(\ast)$.
\end{theorem}

\begin{proof}
Let $I$ be a quasi sdf-absorbing ideal of $R$. By definition, $L=\sqrt{I}$ is
an sdf-absorbing ideal. Since $2\in U(R)$, the sdf-absorbing property of $L$
is equivalent to the condition that $L$ is a prime ideal. Thus, $\sqrt
{I}=\mathfrak{p}$ for some $\mathfrak{p}\in\text{Spec}(R)$. In a
$0$-dimensional Noetherian ring, every prime ideal is a maximal ideal.
Therefore, $\mathfrak{p}=\mathfrak{m}$ for some $\mathfrak{m}\in\text{Max}%
(R)$. We utilize the following result, see \cite[Proposition 4.2]{AM1969}: In
a Noetherian ring if $Q$ is an ideal such that $\sqrt{Q}$ is a maximal ideal,
then $Q$ is primary. Since $\sqrt{I}=\mathfrak{m}$ is maximal, it follows
directly from the above characterization that $I$ is an $\mathfrak{m}$-primary
ideal. Hence $I$ is an sdf-absorbing primary ideal and Condition $(\ast)$ holds.
\end{proof}

\begin{proposition}
\label{P2} Let $R$ satisfy $(\ast)$, let $I,J$ be ideals of $R$, $\sqrt{I}=P$,
such that $I\subseteq J\subseteq P$. If $I$ is $P$-sdf-absorbing primary, then
$J$ is also $P$-sdf-absorbing primary.
\end{proposition}

\begin{proof}
By the monotonicity of the radical operator, the inclusion $I\subseteq
J\subseteq P$ implies $P\subseteq\sqrt{J}\subseteq\sqrt{P}=P$, hence $\sqrt
{J}=P$. Since $I$ is $P$-sdf-absorbing primary, by \cite[Theorem
1]{KhashanCelikel} $P$ is sdf-absorbing, then $J$ is a quasi sdf-absorbing
ideal. Under the assumption that $R$ satisfies condition $(\ast)$, every quasi
sdf-absorbing ideal is sdf-absorbing primary. Thus, $J$ is $P$-sdf-absorbing primary.
\end{proof}

\begin{proposition}
Let $R$ satisfy condition $(\ast)$ and let $\{Q_{i}\}_{i=1}^{n}$ be a finite
family of $P$-sdf-absorbing primary ideals. Then $Q=\bigcap_{i=1}^{n}Q_{i}$ is
a $P$-sdf-absorbing primary ideal.
\end{proposition}

\begin{proof}
The radical of a finite intersection is the intersection of the radicals:
$\sqrt{Q}=\sqrt{\bigcap_{i=1}^{n}Q_{i}}=\bigcap_{i=1}^{n}\sqrt{Q_{i}}=P.$
Since $P$ is an sdf-absorbing ideal, $Q$ is quasi sdf-absorbing by definition.
By condition $(\ast)$, $Q$ is an sdf-absorbing primary ideal with radical $P$.
\end{proof}

\begin{proposition}
Let $R$ and $R^{\prime}$ be rings and $f:R\rightarrow R^{\prime}$ be a
surjective ring homomorphism. If $R$ satisfies condition $(\ast)$, then
$R^{\prime}$ also satisfies condition $(\ast)$.
\end{proposition}

\begin{proof}
Let $Q^{\prime}$ be a quasi sdf-absorbing ideal in $R^{\prime}$. To satisfy
condition $(\ast)$, we must show that $Q^{\prime}$ is an sdf-absorbing primary
ideal. Since $f$ is surjective, we can consider the preimage $Q=f^{-1}%
(Q^{\prime})$, which is an ideal of $R$ containing $\ker(f)$. First, we
observe that $\sqrt{Q}=f^{-1}(\sqrt{Q^{\prime}})$. Since $Q^{\prime}$ is quasi
sdf-absorbing, its radical $\sqrt{Q^{\prime}}$ is an sdf-absorbing ideal in
$R^{\prime}$. It is a known property that the preimage of an sdf-absorbing
ideal under a surjective homomorphism is also sdf-absorbing. Thus, $\sqrt{Q}$
is sdf-absorbing in $R$, which by definition makes $Q$ a quasi sdf-absorbing
ideal in $R$. Since $R$ satisfies condition $(\ast)$, it follows that $Q$ is
an sdf-absorbing primary ideal. Finally, since $f$ is surjective and
$\ker(f)\subseteq Q$, the image $f(Q)=Q^{\prime}$ inherits the sdf-absorbing
primary property, by \cite[Proposition 4 (2)]{KhashanCelikel}. Therefore,
$R^{\prime}$ satisfies condition $(\ast)$.
\end{proof}

\begin{corollary}
Let $R$ be a ring satisfying condition $(\ast)$. For any ideal $J$ of $R$, the
quotient ring $R/J$ also satisfies condition $(\ast)$.
\end{corollary}

\begin{proof}
Consider the canonical projection $\pi:R\rightarrow R/J$, which is a
surjective ring homomorphism. Since $R$ satisfies condition $(\ast)$, it
follows directly from the previous Proposition that its surjective image,
$R/J$, must also satisfy condition $(\ast)$.
\end{proof}

Let $S$ be a multiplicative subset of a ring $R$. An ideal $I$ of $R$ is said
to be $S$-saturated if $I = (S^{-1}I) \cap R$ under the natural homomorphism
$\phi: R \to S^{-1}R$. Equivalently, $I$ is $S$-saturated if for any $s \in S$
and $r \in R$, the condition $sr \in I$ implies $r \in I$.

\begin{lemma}
\label{lemma} Let $S$ be a multiplicative subset of $R$ and $I$ an
$S$-saturated ideal of $R$.

\begin{enumerate}
\item $\sqrt{I}$ is $S$-saturated.

\item Let $S\cap I=\emptyset$. If $S^{-1}I$ is sdf-absorbing in $S^{-1}R$,
then $I$ is sdf-absorbing in $R$.
\end{enumerate}
\end{lemma}

\begin{proof}
(1) Suppose $sx\in\sqrt{I}$ for some $s\in S$ and $x\in R$. Then there exists
$n\in\mathbb{N}$ such that $(sx)^{n}=s^{n}x^{n}\in I$. Since $s\in S$ and $S$
is multiplicatively closed, $s^{n}\in S$. Since $I$ is $S$-saturated,
$s^{n}x^{n}\in I$ implies $x^{n}\in I$, thus $x\in\sqrt{I}$. This proves
$\sqrt{I}$ is $S$-saturated.

(2) Let $0\neq a,b\in R$ such that $a^{2}-b^{2}\in I$. Then $\left(  \frac
{a}{1}\right)  ^{2}-\left(  \frac{b}{1}\right)  ^{2}=\frac{a^{2}-b^{2}}{1}\in
S^{-1}I$. If $\frac{a}{1}=0$, then $ua=0$ for some $u\in S.$ Since $I$ is
$S$-saturated, we have $a=0$, a contradiction. Hence $\frac{a}{1}$ is nonzero
in $S^{-1}R.$ Similarly, one can have that $\frac{b}{1}$ is nonzero. Since
$S^{-1}I$ is sdf-absorbing, we have either $\frac{a-b}{1}\in S^{-1}I$ or
$\frac{a+b}{1}\in S^{-1}I$. This means $s(a-b)\in I$ or $t(a+b)\in I$ for some
$s,t\in S$. As $I$ is $S$-saturated, we conclude $a-b\in I$ or $a+b\in I$, as needed.
\end{proof}

\begin{theorem}
If a ring $R$ satisfies condition $(*)$, then $S^{-1}R$ satisfies condition
$(*)$ for any multiplicative subset $S \subseteq R$.
\end{theorem}

\begin{proof}
Let $S^{-1}I$ be a quasi sdf-absorbing ideal in $S^{-1}R$. Then $S^{-1}I$ is
proper, and so $S\cap I=\emptyset.$ To prove that $S^{-1}R$ satisfies
condition $(\ast)$, we must show that $S^{-1}I$. is an sdf-absorbing primary
ideal of $S^{-1}R.$ Since $S^{-1}\sqrt{I}=\sqrt{S^{-1}I}$ is sdf-absorbing
ideal, $\sqrt{I}$ is sdf-absorbing in $R$ by Lemma \ref{lemma}. Hence, $I$ is
quasi sdf-absorbing in $R$. Since $R$ is assumed to satisfy condition $(\ast
)$, it follows that $I$ is sdf-absorbing primary. By \cite[Proposition 3
(1)]{KhashanCelikel}, its extension $S^{-1}I$ is an sdf-absorbing primary
ideal in $S^{-1}R$ and therefore, $S^{-1}R$ satisfies condition $(\ast).$
\end{proof}

\begin{theorem}
Let $R$ be a ring and $M$ be an $R$-module. If $R$ satisfies condition
$(\ast)$, then the idealization $R \ltimes M$ satisfies condition $(\ast)$ for
all ideals of the form $I \ltimes M$.
\end{theorem}

\begin{proof}
Let $\mathcal{I} = I \ltimes M$ be an ideal of the ring $\mathcal{R} = R
\ltimes M$. Then $\sqrt{I \ltimes M} = \sqrt{I} \ltimes M$. Suppose that
$\mathcal{I}$ is a \text{quasi sdf-absorbing} ideal in $\mathcal{R}$. By
definition, its radical $\sqrt{I} \ltimes M$ is an sdf-absorbing ideal in
$\mathcal{R}$. From \cite[Theorem 4.19]{AB2024} $\sqrt{I}$ is an sdf-absorbing
ideal in $R$. This implies that $I$ is a \text{quasi sdf-absorbing} ideal of
$R$. Since $R$ satisfies condition $(\ast)$, the ideal $I$ is necessarily
\text{sdf-absorbing primary} and so $\mathcal{I} = I \ltimes M$ is
sdf-absorbing primary in $\mathcal{R}$, by \cite[Theorem 8 (1)]%
{KhashanCelikel}. We conclude that $R \ltimes M$ inherits condition $(\ast)$
for this class of ideals.
\end{proof}

\begin{theorem}
Let $K$ be a field. The polynomial ring $K[x]$ always satisfies Condition
$(\ast)$.
\end{theorem}

\begin{proof}
If $char(K)=2,$ then the claim is clear by Example \ref{estar}(2). So, suppose
that $char(K)\neq2$. Let $I\subseteq K[x]$ be a quasi sdf-absorbing ideal. By
definition, its radical $L=\sqrt{I}$ is an sdf-absorbing ideal, meaning that
for any $a,b\in K[x]$, if $a^{2}-b^{2}\in L$, then $(a-b)\in L$ or $(a+b)\in
L$. We aim to show that $I$ is a primary ideal. \newline Since $K[x]$ is a
Principal Ideal Domain (PID), every ideal is generated by a single polynomial.
Let $I=(f)$ and $L=\sqrt{I}=(p)$. In a PID, the radical of an ideal generated
by $f$ is generated by the product of the distinct irreducible factors of $f$.
Thus, we can write $p=p_{1}p_{2}\dots p_{n},$ where $p_{i}$ are distinct
irreducible polynomials in $K[x]$. We aim to show $n=1$. Suppose $n>1$. We use
the Chinese Remainder Theorem to construct a polynomial $g$ such that:%
\[%
\begin{cases}
g\equiv1\pmod{p_1}\\
g\equiv-1\pmod{p_2 \dots p_n}
\end{cases}
\]
This is possible because $p_{1}$ and the product $p_{2}\dots p_{n}$ are
relatively prime. Let $a=g$ and $b=1$. Then $a^{2}-b^{2}=(g-1)(g+1)$. Since
$g\equiv1\pmod{p_1}$, we have $p_{1}\mid(g-1)$. Also, since $g\equiv
-1\pmod{p_2 \dots p_n}$, we have $(p_{2}\dots p_{n})\mid(g+1)$. Now, since
$p_{1}$ and $p_{2}\dots p_{n}$ are coprime, then $p\mid(g-1)(g+1)$ which means
$g^{2}-1\in(p)=L.$ Since $L$ is sdf-absorbing, either $(g-1)\in L$ or
$(g+1)\in L$. Hence, we have the following cases:

Case I: If $(g-1)\in L=(p_{1}\dots p_{n})$, then $p_{2}\mid(g-1)$. But our
construction says $g\equiv-1\pmod{p_2}$, so $g-1\equiv-2\pmod{p_2}$. As
$\text{char}(K)\neq2$, this implies $p_{2}$ is a unit, a contradiction.

Case II. If $(g+1)\in L$, then $p_{1}\mid(g+1)$. But $g\equiv1\pmod{p_1}$, so
$g+1\equiv2\pmod{p_1}$, again implying $p_{1}$ is a unit, a contradiction.

The assumption $n>1$ leads to a contradiction. Thus $n=1$, meaning $L=(p_{1})$
is prime. This ensures $I$ is a primary ideal, and specifically, an
sdf-absorbing primary ideal.
\end{proof}

\end{document}